%% file: lyap_rank1.tex
\documentclass[11pt,letterpaper]{article}

\input{lyap_rank1_preamble}

\usepackage[margin=1in]{geometry}

\makeatletter
\def\blfootnote{\gdef\@thefnmark{}\@footnotetext}
\makeatother

\begin{document}

\title{Lyapunov Exponent of Rank One Matrices: Ergodic Formula and Inapproximability of the Optimal Distribution}
\author{Jason M. Altschuler \and Pablo A. Parrilo}
\date{}
\maketitle

\blfootnote{The authors are with the Laboratory for Information and Decision Systems (LIDS), Massachusetts Institute of Technology, Cambridge MA 02139. Work partially supported by NSF AF 1565235 and NSF Graduate Research Fellowship 1122374.} 

\begin{abstract} 

The \emph{Lyapunov exponent} corresponding to a set of
  square matrices $\cA = \{A_1, \dots, A_n \}$ and a probability
  distribution $p$ over $\{1, \dots, n\}$ is $\lyap := \lim_{k \to
    \infty} \frac{1}{k} \,\E \log \norm{A_{\sigma_k} \cdots
    A_{\sigma_2}A_{\sigma_1}}$, where $\sigma_i$ are i.i.d. according
  to $p$. This quantity is of fundamental importance to control theory
  since it determines the asymptotic convergence rate $e^{\lyap}$ of the
  stochastic linear dynamical system $x_{k+1} = A_{\sigma_k}
  x_k$. This paper investigates the following ``design problem'': given
  $\cA$, compute the distribution $p$ minimizing $\lyap$. Our main
  result is that it is $\NPclass$-hard to decide whether there exists
  a distribution $p$ for which $\lyap < 0$, i.e. it is $\NPclass$-hard
  to decide whether this dynamical system can be stabilized.


This hardness result holds even in the ``simple'' case where $\cA$ contains only rank-one matrices. Somewhat surprisingly, this is in stark contrast to the Joint Spectral Radius -- the deterministic kindred of the Lyapunov exponent -- for which the analogous optimization problem for rank-one matrices is known to be exactly computable in polynomial time.



To prove this hardness result, we first observe via Birkhoff's
Ergodic Theorem that the Lyapunov exponent of rank-one matrices admits
a simple formula and in fact is a quadratic form in~$p$. Hardness of
the design problem is shown through a reduction from the Independent
Set problem. Along the way, simple examples are given illustrating
that $p \mapsto \lyap$ is neither convex nor concave in general. We conclude with extensions to continuous distributions, exchangeable processes, Markov processes, and stationary ergodic processes.
\end{abstract}


\section{Introduction}\label{sec:intro}
Given a finite collection of square matrices $\cA = \{A_1, \dots, A_n \}$
 and a probability distribution $p$ over $\{1, \dots, n\}$, the corresponding \emph{Lyapunov exponent} $\lyap$
is defined as:
\begin{align}
\lyap := \lim_{k \to \infty} \frac{1}{k} \,\E \log \norm{A_{\sigma_k} \cdots A_{\sigma_2}A_{\sigma_1} },
\label{eq:def:lyap}
\end{align}
where the $\sigma_i$ are independently and identically distributed (i.i.d.) according to $p$. 
Over the past half century, this quantity $\lyap$ has received significant interest due to its many connections and applications to diverse fields including probability, ergodic theory, dynamical systems, functional analysis, representation theory, computer image generation of fractals, control theory, and more; see e.g.\  the surveys~\citep{CohKesNew86,Cohen98,DiaFre99,Furman02,Viana14,BenQui16} and references within. 
\par The primary application of the Lyapunov exponent that this paper focuses on is in control theory. The fundamental connection is that $\lyap$ dictates the asymptotic growth rate of the stochastic linear dynamical system
\begin{align}
x_{k+1} = A_{\sigma_k} x_k.
\label{eq:lds}
\end{align}
Indeed, by the celebrated Furstenberg-Kesten Theorem~\citep{FurKes60}, the average norm $\lim_{k \to \infty} \|A_{\sigma_k}\cdots A_{\sigma_2}A_{\sigma_1}\|^{1/k}$ almost surely (a.s.) exists and is equal to the \emph{Lyapunov spectral radius}
\begin{align}
\explyap := e^{\lyap}.
\label{eq:def:explyap}
\end{align}
In particular, the dynamical system~\eqref{eq:lds} converges asymptotically\footnote{
The Lyapunov exponent characterizes a.s.\ convergence of~\eqref{eq:lds}. In general, this is different from convergence in \emph{mean}, i.e., $\lim_{k \to \infty} \E\|A_{\sigma_k} \cdots A_{\sigma_1} \| = 0$. For example, consider the scalar case with $n=2$, where $A_{\sigma_k}$ are i.i.d., and are either $2$ or $\tfrac{1}{100}$ with probability $\half$. Then $A_{\sigma_k}\cdots A_{\sigma_1}$ does not converge in mean (in fact $\E|A_{\sigma_k} \cdots A_{\sigma_1}|$ grows exponentially in $k$) due to a few exponentially large but exponentially unlikely trajectories, but \emph{does} converge to $0$ with probability $1$ (since $\exp(\E \log |A_{\sigma_1}| ) < 1$).} to zero (a.s.) if and only if $\explyap < 1$, i.e. if and only if $\lyap < 0$. Note that in this sense, the Lyapunov spectral radius is a natural stochastic generalization of the spectral radius of a single matrix. Two natural control-theoretic questions arise about the dynamical system~\eqref{eq:lds}: analysis and design.
\par The \emph{analysis problem} is to compute the convergence rate $\lyap$ of the dynamical system given (fixed) parameters $(\cA,p)$. This problem is well-known to be difficult, and approximating $\lyap$ is in general algorithmically undecidable~\citep[Theorem 2]{TsiBlo97}.
Nevertheless, the significance of this problem has motivated a substantial body of research on algorithms for obtaining upper or lower bounds on $\lyap$, and on identifying special cases in which this problem becomes tractable; see e.g.~\citep{DerrMecPic87, DeuPal89, Key90, Mainieri92, DieVan02, Protasov04, GhaAna05, DieVan06, JunProtBlo09, BeyLus09, ProJun13}. 
\par The \emph{design problem} -- which is the focus of this paper -- is to optimize the dynamical system so that the convergence is as fast as possible, namely: given a fixed set of matrices $\cA$, compute a distribution $p$ minimizing $\lyap$. In contrast to the analysis question, there seems to be little previous work on this design problem. (A related version of this design problem uses instead the more restrictive notion of convergence in \emph{mean}, see e.g.~\citep{Boyd06}.)
One of this paper's primary contributions is to explain this lack of progress on the design problem by showing that it is $\NPclass$-hard to determine whether there exists a distribution $p$ such that $\lyap < 0$. That is, unless $\Pclass = \NPclass$, there is no polynomial time algorithm for deciding whether the dynamical system~\eqref{eq:lds} has a stabilizing distribution $p$. Moreover, this complexity result holds even in the ``simple'' case where all matrices in $\cA$ are rank one.

\subsection{Contributions and outline}
This paper focuses on the case when all matrices in $\cA$ are \emph{rank one}. The advantage of restricting to this simple case is that (i) the analysis problem becomes tractable as $\lyap$ admits a simple, closed-form expression; and (ii) this formula enables us to prove that the design problem is hard in a complexity-theoretic sense. We elaborate below on details and implications. Table~\ref{tab:overview} contains an overview of our main results. 
\par Section~\ref{sec:formula} studies the analysis problem when $\cA$ is a collection of rank-one matrices. We observe in Theorem~\ref{thm:formula} that in this special case, $\lyap$ admits a simple closed-form expression and is in fact a \emph{quadratic form} in $p$.
This is an ergodic-type result since it equates $\lyap$ -- which is a ``time average'' of the product of an infinite sequence of random matrices -- to a simpler ``space average.'' We note that this formula readily extends to ergodic stationary processes taking values in sets $\cA$ consisting of infinitely many rank-one matrices; however, for expository reasons we defer this generalization to Section~\ref{sec:extensions}, and provide in Section~\ref{sec:formula} just the simple case of i.i.d. processes taking values in finite sets $\cA$, since this admits an elementary proof and requires essentially no measure theory. We remark that this formula shows that for this rank-one case, $\lyap$ is computable in polynomial time (as opposed to being undecidably hard to approximate in the general case~\citep{TsiBlo97}).

\begin{table}[t!]
	\begin{tabular}{|c|c|c|c|c|}
		\hline
		& \textbf{Analysis question}                                                                                   & \textbf{\begin{tabular}[c]{@{}c@{}}Design question\\ (i.i.d.)\end{tabular}} & \textbf{\begin{tabular}[c]{@{}c@{}}Design question\\ (exchangeable)\end{tabular}} & \textbf{\begin{tabular}[c]{@{}c@{}}Design question\\ (Markov)\end{tabular}}                 \\ \hline
		\textbf{Rank one} & \textit{\begin{tabular}[c]{@{}c@{}}Poly-time algorithm\\ (Theorem 2.1)\end{tabular}}                  & \textit{\begin{tabular}[c]{@{}c@{}}NP-hard\\ (Theorem 4.1)\end{tabular}}    & \textit{\begin{tabular}[c]{@{}c@{}}NP-hard\\ (Theorem 5.2)\end{tabular}}          & \textit{\begin{tabular}[c]{@{}c@{}}Poly-time algorithm\\ (Theorem 5.3)\end{tabular}} \\ \hline
		\textbf{General}  & \begin{tabular}[c]{@{}c@{}}Undecidable\\ ~\citep[Theorem 2]{TsiBlo97} \end{tabular} & \textit{\begin{tabular}[c]{@{}c@{}}NP-hard\\ (Theorem 4.1)\end{tabular}}    & \textit{\begin{tabular}[c]{@{}c@{}}NP-hard\\ (Theorem 5.2)\end{tabular}}          & ?                                                                                           \\ \hline
	\end{tabular}
	\caption{Complexity overview of the analysis and design questions for the Lyapunov exponent -- for collections of general matrices, as well as the special case of rank-one matrices. The various versions of the design problem are for optimizing over different classes of stochastic processes. Our new results are shown in italics.}
	\label{tab:overview}
\end{table}

\par Sections~\ref{sec:convexity} and~\ref{sec:inapprox} then turn to the design problem. Specifically, Section~\ref{sec:convexity} investigates the convexity and concavity of the function $p \mapsto \lyap$ over the probability simplex $\Delta_n$. By our aforementioned result that $\lyap$ is a quadratic form in this rank-one case, these convexity/concavity properties are equivalent to the (conditional) definiteness of the Hessian $\tfrac{\partial^2}{\partial p^2} \lyap$. We give simple, small examples illustrating that in general this matrix is neither (conditionally) positive semidefinite nor negative semidefinite, and then explain this phenomenon by making a connection to the fact that the Martin distance on the $(1,n)$ Grassmanian is not a metric.
\par Of course, this lack of convexity of the function $p \mapsto \lyap$ does not in itself prove that the design problem $\min_{p \in \Delta_n} \lyap$ is hard. Section~\ref{sec:inapprox} formally shows that it is $\NPclass$-hard to even determine the sign of $\lyapmin$. That is, unless $\Pclass = \NPclass$, there is no polynomial-time algorithm for determining whether the dynamical system~\eqref{eq:lds} has a stabilizing distribution $p$. 
This explains the lack of prior algorithmic progress on the design problem. The key idea in the proof is a reduction from the Independent Set problem, which is known to be $\NPclass$-hard to approximate and is also known to be expressible as the minimization of a certain quadratic form over the probability simplex. The intuition for this reduction is, roughly speaking, that the optimal distribution for $\lyap$ should concentrate mass on a set of matrices in $\cA$ which are as mutually ``orthogonal'' to each other as possible, and such a set of matrices can be thought of as an independent set in a certain associated graph. (See the proof of Theorem~\ref{thm:dec} for precise details.)
\par These hardness results uncover an interesting and somewhat surprising (at least to the authors) dichotomy between the complexity of optimizing the Lyapunov exponent and its deterministic analogue, the \emph{Joint Spectral Radius} (JSR). Recall that the JSR of a collection of square matrices $\cA = \{A_1, \dots, A_n \}$ is defined as:
\begin{align}
\jsr := \lim_{k \to \infty} \max_{\sigma \in \{1,\dots,n\}^k}  \norm{A_{\sigma_k} \cdots A_{\sigma_2}A_{\sigma_1} }^{1/k}.
\label{eq:def:jsr}
\end{align}
The JSR is a natural generalization of the spectral radius of a single matrix and characterizes the maximal (i.e., worst-case) asymptotic growth rate of the linear dynamical system~\eqref{eq:lds} over all possible choices of $A_{\sigma_k}$. Due to its many applications (see e.g.~\citep{Jungers09} and the references within), the JSR has also attracted a large body of research. Similarly to the Lyapunov exponent, the JSR is known to in general have algorithmic undecidability and $\NPclass$-hardness complexity barriers~\citep{TsiBlo97, BloTsi00}, yet nevertheless there has been significant work on approximation algorithms for $\jsr$ and on identifying special cases in which the problem becomes more tractable~\citep{ParJad08, ProJunBlo10, GugPro13, AhmJunParRoo14, Jungers09}. In particular, it is known that (optimizing) the JSR of rank-one matrices is equivalent to a certain discrete optimization problem over an associated graph, namely the Maximum-Cycle-Mean problem, which can be solved in polynomial time~\citep{GurSam05,AhmPar12,LiuXia13,Karp78}. This is in stark contrast to our result that the Lyapunov exponent of rank-one matrices is $\NPclass$-hard to optimize, even approximately. 
\par In Subsection~\ref{subsec:ptas}, we note that in the special case when the set $\cA$ is ``well-conditioned,'' there is a polynomial-time approximation algorithm for $\lyapmin$.
\par Finally, Section~\ref{sec:extensions} considers extensions to more general stochastic processes $A_{\sigma_1}$, $A_{\sigma_2}$, $\dots$. The corresponding analysis and design problems are discussed in Subsection~\ref{subsec:extensions:stoc} and~\ref{subsec:extensions:des}, respectively. Specifically, Subsection~\ref{subsec:extensions:stoc} shows that a similar formula holds for the ``Lyapunov exponent'' of ergodic stationary processes of rank-one matrices. This allows for the study of more exotic stochastic linear dynamical systems~\eqref{eq:lds}. Using this formula, we compute the convergence rate for setups including i.i.d. processes over continuous distributions, exchangeable processes, and Markov processes. As an example application, we explicitly calculate the Lyapunov exponent for the induced distribution over matrices $uu^T$, where $u$ is distributed according to the $d$-dimensional spherical measure. Subsection~\ref{subsec:extensions:des} then turns to the design problem for these more general setups. We show that optimizing the convergence rate of the dynamical system~\eqref{eq:lds} is $\NPclass$-hard over exchangeable processes, but can be done in polynomial time for Markov processes. The intuition for these results is as follows. First, we observe that (after some manipulation) these problems seek to optimize a linear functional over a convex set, and thus are equivalent to optimizing over the extreme points. Informally speaking, the extreme points of the sets of exchangeable processes and Markov processes, respectively, are the sets of i.i.d. process and deterministic processes. The arguments are then finished by recalling that optimizing the convergence rate over i.i.d. processes is $\NPclass$-hard (discussed above), while optimizing over deterministic processes is known to be polynomial-time computable.

\subsection{Notation}
The set of discrete probability distributions on $n$ atoms is identified with the simplex $\Delta_n := \{p \in \Rpn : \sum_{i=1}^n p_i = 1 \}$. The support of a discrete probability distribution $p \in \Delta_n$ is denoted by $\supp(p) := \{i \in [n] : p_i > 0 \}$. We write $X \sim p$ to denote that $X$ is a random variable with distribution $p$. The symbols $e_i$, $\bone$, and $I$ denote the $i$th standard basis vector, all-ones vector, and identity matrix, respectively, in an ambient dimension that will be clear from context. The norm $\|\cdot\|_2$ denotes the Euclidean norm when applied to vectors, and denotes the spectral norm when applied to matrices.
All logarithms are natural logarithms, i.e., taken base $e$. We abbreviate ``almost surely'' by a.s., ``without loss of generality'' by w.l.o.g., ``independent and identically distributed'' by i.i.d., ``positive semidefinite'' by PSD, and ``negative semidefinite'' by NSD. Recall that a square matrix $M$ is said to be conditionally PSD (resp. conditionally NSD) if $p^TMp \geq 0$ (resp. $\leq 0$) for all vectors $p$ in the orthogonal component of the subspace spanned by $\bone$. All other specific notations are introduced in the main text.

\section{Ergodic formula for Lyapunov exponent of rank one matrices}\label{sec:formula}

This section gives an explicit closed-form expression for the Lyapunov exponent of a collection of rank-one matrices. This is an ergodic-type result in the sense that it equates the Lyapunov exponent -- which is a ``time average'' -- to a simpler ``space average.''
\par This result readily extends to stationary ergodic processes where the collection $\cA$ contains infinitely many matrices (Theorem~\ref{thm:stoc}). However for simplicity of exposition, we first present the discrete case (Theorem~\ref{thm:formula}) since it admits an elementary proof with essentially no measure theory.

\begin{theorem}\label{thm:formula}
For any collection of rank-one square matrices $\cA := \{A_i = u_iv_i^T \}_{i=1}^n$ and any probability distribution $p \in \Delta_n$,
\begin{align}
\lyap = \sum_{i,j=1}^n p_i p_j \log \abs{u_i^Tv_j}.
\label{eq:thm:formula}
\end{align}
\end{theorem}

\begin{remark}\label{rem:skew-symm}
	The formula in Theorem~\ref{thm:formula} is independent of the rank-one decompositions $A_i = u_iv_i^T$. Indeed, a rescaling of $u_i$ and $v_i$ changes the cost matrix $[\log |u_i^Tv_j|]_{ij}$ by a skew-symmetric matrix, which does not affect the underlying quadratic form $\sum_{ij} p_i p_j \log \abs{u_i^Tv_j}$.
\end{remark}

It is important to emphasize that, in general, the Lyapunov exponent of a set of matrices \emph{cannot} be computed in closed form. It is well-known that (under mild conditions) the Lyapunov exponent is equal to a simpler ``space average'' over a certain invariant distribution~\citep{Furstenberg63}; however, the complexity of such ergodic formulas lies in the invariant distribution, which is often difficult -- or impossible -- to characterize or manipulate analytically. The key difference for the rank-one setting studied here is that in this special case, the invariant distribution has a simple closed form. In words, this is because in this rank-one case, the contraction of the next matrix $A_{\sigma_{k+1}}$ depends on the previous matrices essentially only through the last matrix $A_{\sigma_k}$.

\begin{proof}[Proof of Theorem~\ref{thm:formula}]
Note that the definition of $\lyap$ is independent of the norm $\|\cdot\|$; for convenience of computation, let us take it to be the spectral norm. Let $\sigma_1, \sigma_2, \dots$ be i.i.d. according to $p$, and expand the (random) quantity
\begin{align}
\frac{1}{k} \log \norm{A_{\sigma_k} \cdots A_{\sigma_2} A_{\sigma_1} }_2
&=
\frac{1}{k} \log \norm{(u_{\sigma_k}v_{\sigma_k}^T) \cdots (u_{\sigma_2}v_{\sigma_2}^T) (u_{\sigma_1}v_{\sigma_1}^T) }_2
\nonumber
\\ &= \frac{1}{k} \sum_{t=1}^{k-1} \log \abs{u_{\sigma_t}^Tv_{\sigma_{t+1}}}
+ \frac{1}{k} \left(\log \|u_{\sigma_k}\|_2 + \log \|v_{\sigma_1}\|_2\right).
\label{eq:pf-thm-formula:expand}
\end{align}
We may assume w.l.o.g. that $A_i \neq 0$ for all $i \in \supp(p)$, since otherwise the proof is immediate as both sides of~\eqref{eq:thm:formula} equal $-\infty$. Thus the latter term $\frac{1}{k} (\log \|u_{\sigma_k}\|_2 + \log \|v_{\sigma_1}\|_2)$ in~\eqref{eq:pf-thm-formula:expand} converges to $0$ as $k \to \infty$. 
\par Now let $G_{\cA}$ be the complete directed graph on $n$ nodes, one for each matrix $A_i$, with weight $\log |u_i^Tv_j|$ on the edge from node $i$ to node $j$. Consider the $n$-state Markov chain on $G_{\cA}$ with transition probability $p_j$ from node $i$ to node $j$, irrespective of $i$. The key point is that after multiplying by $\tfrac{k}{k-1}$, the sum $\frac{1}{k-1} \sum_{t=1}^{k-1} \log |u_{\sigma_t}^Tv_{\sigma_{t+1}}|$ in~\eqref{eq:pf-thm-formula:expand} is precisely the average weight of edges travelled when running this Markov chain for $k-1$ steps from initial state $\sigma_1 \sim p$. Now since this Markov chain has stationary distribution $p$, a standard ergodic result on discrete-time, finite-state Markov chains (see e.g.~\citep{GriSti01}) shows that this quantity converges a.s. to $\sum_{ij} p_i p_j \log |u_i^Tv_j|$. Thus by~\eqref{eq:pf-thm-formula:expand} and the fact $\lim_{k \to \infty} \tfrac{k-1}{k} = 1$, we conclude that
\begin{align}
\lim_{k \to \infty} \frac{1}{k} \log \norm{A_{\sigma_k} \cdots A_{\sigma_2} A_{\sigma_1} } \aseq
\sum_{ij} p_i p_j \log |u_i^Tv_j|.
\label{eq:pf-thm-formula:as}
\end{align}
By sub-multiplicativity of the spectral norm, $ |\frac{1}{k}\log \norm{A_{\sigma_k} \cdots A_{\sigma_2} A_{\sigma_1} }_2| \leq \max_{i \in \supp(p)}|\log \|A_i\|_2|$ is uniformly bounded for all $k \in \N$. Thus the claimed formula for $\lyap$ follows from the a.s. convergence proven in~\eqref{eq:pf-thm-formula:as} and an application of Lebesgue's Dominated Convergence Theorem. 
\end{proof}

\begin{remark}\label{rem:product-two}
Theorem~\ref{thm:formula} shows that the Lyapunov exponent of rank-one matrices depends only on the products $A_iA_j$ of length \emph{two}. This is in contrast to the Joint Spectral Radius of rank-one matrices, which depends on all matrix products of length up to $n$~\citep[Corollary 2.1]{AhmPar12}.
\end{remark}

\begin{remark}\label{rem:computable}
Theorem~\ref{thm:formula} shows that the Lyapunov exponent of rank-one matrices is not only computable, but also computable in polynomial time. This is in sharp contrast to the Lyapunov exponent of general matrices (i.e., not necessarily rank-one), which is \emph{undecidably hard} to approximate~\citep[Theorem 2]{TsiBlo97}.
\end{remark}

\begin{remark}
It is an interesting open question whether such results extend to collections of matrices with some small fixed rank $k > 1$. In particular, is there a closed-form expression for the Lyapunov exponent, analogously to Theorem~\ref{thm:formula}? And if so, how do the phenomena mentioned in Remark~\ref{rem:product-two} and~\ref{rem:computable} change?
\end{remark}


\section{Convexity properties of the Lyapunov exponent as a function of the distribution}\label{sec:convexity}

From Theorem~\ref{thm:formula}, we know that $\lyap$ is a quadratic form in $p$:
\begin{align}
\lyap = p^T \MA p,
\label{eq:thm:formula-quad}
\end{align}
where $(\MA)_{ij} := \half (\log |u_i^Tv_j| + \log |u_j^Tv_i|)$.  This resolves the \emph{analysis problem} of computing the Lyapunov exponent given a collection of rank-one matrices $\cA$ and a probability distribution $p$. In this and the following section, we now turn to the \emph{design problem}: Given a collection of matrices $\cA$, find the distribution $p \in \Delta_n$ minimizing $\lyap$. As mentioned in the introduction, this is a fundamental problem in control theory with applications including e.g., fast convergence in the corresponding stochastic linear dynamical system~\eqref{eq:lds}. 
\par We begin with an example illustrating that even in simple cases, the optimal distribution can depend in a non-obvious way on the geometric configuration of the matrices in $\cA$.

\begin{figure}[t!]
  \begin{subfigure}[t]{.33\textwidth}
    \centering
 \begin{tikzpicture}
     \draw[<->] (-1,0)--(3,0);
     \draw[<->] (0,-1)--(0,3);
   \draw[line width=3pt,blue,-stealth](0,0)--(2,0.2) node[anchor=south west]{$\boldsymbol{u_2 = (10,1)}$};
   \draw[line width=3pt,blue,-stealth](0,0)--(0.2,2) node[anchor=south west]{$\boldsymbol{u_1 = (1,10)}$};
  \draw[line width=1pt,blue,-stealth](0,0)--(1.6,1.6) node[anchor=south west]{$\boldsymbol{u_3 = (8,8)}$};
 \end{tikzpicture}
 \caption{$p^* = (\half,\half,0)$, $\lambda(\cA, p^*) \approx 3.805$.}
 \label{fig:min:left}
  \end{subfigure}%
  \begin{subfigure}[t]{.33\textwidth}
    \centering
\begin{tikzpicture}
    \draw[<->] (-1,0)--(3,0);
    \draw[<->] (0,-1)--(0,3);
 \draw[line width=1pt,blue,-stealth](0,0)--(2,0.2) node[anchor=south west]{$\boldsymbol{u_2 = (10,1)}$};
  \draw[line width=1pt,blue,-stealth](0,0)--(0.2,2) node[anchor=south west]{$\boldsymbol{u_1 = (1,10)}$};
  \draw[line width=3pt,blue,-stealth](0,0)--(0.6,0.6) node[anchor=south west]{$\boldsymbol{u_3 = (3,3)}$};
\end{tikzpicture}
   \caption{$p^* = (0,0,1)$, $\lambda(\cA, p^*) \approx 2.890$.}
   \label{fig:min:mid}
  \end{subfigure}
   \begin{subfigure}[t]{.33\textwidth}
      \centering
 \begin{tikzpicture}
     \draw[<->] (-1,0)--(3,0);
     \draw[<->] (0,-1)--(0,3);
   \draw[line width=3pt,blue,-stealth](0,0)--(2,0.2) node[anchor=south west]{$\boldsymbol{u_2 = (10,1)}$};
   \draw[line width=3pt,blue,-stealth](0,0)--(0.2,2) node[anchor=south west]{$\boldsymbol{u_1 = (1,10)}$};
   \draw[line width=3pt,blue,-stealth](0,0)--(1.3,-0.7) node[anchor=south west]{$\boldsymbol{u_3 = (6.5,-3.5)}$};
 \end{tikzpicture}
      \caption{$p^* \approx (0.156,0.377,0.467)$,  $\lambda(\cA, p^*) \approx 3.772$.}
    \label{fig:min:right}
    \end{subfigure}
        \caption{Vectors $\{u_i\}_{i=1}^3$ for the three scenarios in Example~\ref{ex:min}. Bolded vectors correspond to the support of the optimal distribution $p^* := \argmin_{p \in \Delta_3} \lambda( \cA,p)$, where $\cA := \{u_iu_i^T \}_{i=1}^3$.}
        \label{fig:min}
\end{figure}
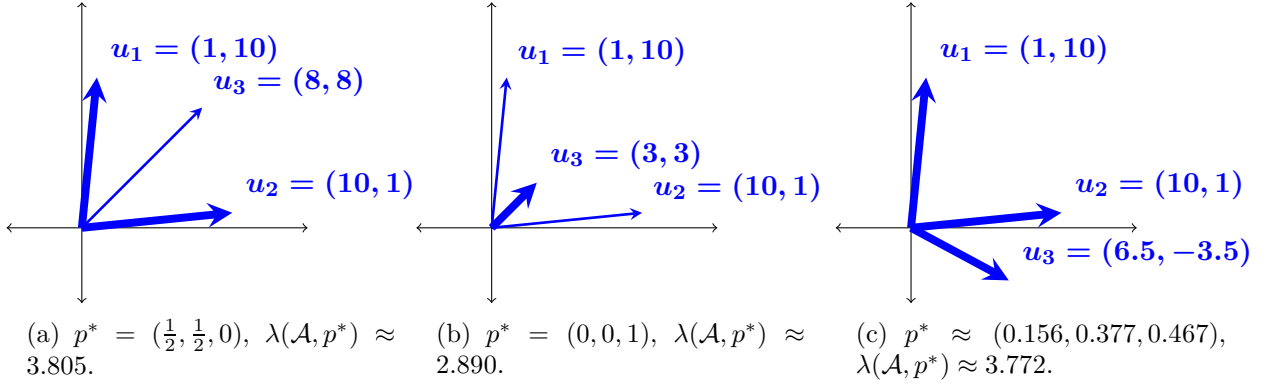

\begin{example}\label{ex:min}
We consider several cases of symmetric rank-one matrices $\cA = \{u_iu_i^T\}_{i=1}^3$ (plotted in Figure~\ref{fig:min}) to illustrate the design problem $\lyapmin$. This problem has a simple and intuitive geometric interpretation in terms of the linear dynamical system\footnote{This can be viewed e.g., as a randomized instance of Dykstra's alternating-projection method for finding a point at the intersection of convex sets (namely the lines $u_i$ here)~\citep{Dykstra83}.} $x_{k+1} = A_{\sigma_k}x_k = u_{\sigma_k}u_{\sigma_k}^T x_k$, which corresponds to projecting (and scaling) the current state vector $x_k$ onto a randomly chosen vector $u_{\sigma_k}$. Specifically, $\lyapmin$ corresponds to finding the distribution $p \in \Delta_3$ over lines $\{u_i\}_{i=1}^3$ such that repeatedly projecting (and scaling) onto a random vector $u_i$ with probability $p_i$ converges as quickly as possible. In all examples, $u_1 = [1,10]^T$ and $u_2 = [10,1]^T$ are close to orthogonal in that they have small inner product; this is why the optimal distribution is concentrated on them in Figure~\ref{fig:min:left}.
In Figure~\ref{fig:min:mid}, $u_3$ has smaller norm, and the optimal distribution is concentrated just on it. In Figure~\ref{fig:min:right}, the optimal distribution is supported on all three matrices.
Finally, we note that the optimal distribution for the \emph{maximization} problem $\lyapmax$ is easily characterized: it is always concentrated on the vector(s) of largest norm (see Lemma~\ref{lem:min-max} below).
\end{example}

In the rest of the section, we investigate the (lack of) convexity and concavity of the function $p \mapsto \lyap$ for a fixed collection of rank-one matrices $\cA = \{u_iv_i^T\}_{i=1}^n$. Algorithmic and complexity results for the design problem $\lyapmin$ are then discussed in the following section.
\par Note that since $p \in \Delta_n$ is a probability distribution in our context, the function $\Delta_n \ni p \mapsto \lyap = p^T\MA p$ is convex (resp. concave) if and only if the matrix $\MA$ is conditionally PSD (resp. conditionally NSD).
We begin by observing that these conditional semidefiniteness properties of $\MA$ are invariant under rescaling of the matrices in $\cA$.

\begin{lemma}\label{lem:rescale}
The conditional semidefiniteness of $\MA$ is invariant under non-zero rescalings of matrices in $\cA$. 
\end{lemma}
\begin{proof}
It suffices to show this invariance holds under rescaling $u_i$ by some $\alpha \neq 0$. This scaling changes $\MA$ to $\MA + \Delta$ where $\Delta := \half \log|\alpha| (e_i\bone^T + \bone e_i^T)$. Now for any vector $x$ orthogonal to $\bone$, $x^T\Delta x = 0$, thus in particular $x^T(\MA + \Delta) x = x^T \MA x$. Therefore $\MA + \Delta$ is conditionally PSD (resp. NSD) if and only if $\MA$ is.
\end{proof}

As seen next, $\MA$ is neither conditionally PSD nor conditionally NSD in general. This is true even if all matrices in $\cA$ are symmetric, i.e. when each $u_i = v_i$, so the rest of this section henceforth restricts to this symmetric case $\cA = \{u_iu_i^T\}_{i=1}^n$ for simplicity of exposition.

\begin{example}\label{exa:not-nsd}
Let $\cA = \{ e_1e_1^T, e_2e_2^T \} \subset \R^{2 \times 2}$, where $e_1 := [1,0]^T$ and $e_2 := [0,1]^T$. Then $\MA = \begin{bmatrix}
0 & -\infty \\
-\infty & 0
\end{bmatrix}$ is not conditionally NSD since $\begin{bmatrix} 1 & -1 \end{bmatrix} \MA \begin{bmatrix} 1 \\ -1 \end{bmatrix} = \infty \not \leq 0$.
\end{example}

\begin{example}\label{exa:not-psd}
Let $\cA = \{e_1e_1^T, e_2e_2^T, (e_1+e_2)(e_1+e_2)^T \} \subset \R^{2 \times 2}$, where $e_1 := [1,0]^T$ and $e_2 := [0,1]^T$. Then $\MA = 
\begin{bmatrix}
0 & -\infty & 0 \\
-\infty & 0 & 0 \\
0 & 0 & \log 2
\end{bmatrix}
$
is not conditionally PSD since $\begin{bmatrix} 1 & 1 & -2 \end{bmatrix} \MA \begin{bmatrix} 1 \\ 1 \\ -2 \end{bmatrix} = -\infty \not \geq 0$.
\end{example}

Note that Examples~\ref{exa:not-nsd} and~\ref{exa:not-psd} are easily modified so that $\MA$ only has finite entries: add $\eps \bone$ to each of the vectors $u_i$, for some small $\eps > 0$.

We conclude the section by connecting this lack of convexity of $p \mapsto \lyap$ to the Martin distance on the space of lines (i.e., the $(1,n)$ Grassmanian). For simplicity, we continue to restrict to the symmetric case $A_i = u_iu_i^T$ and assume that the vectors $u_i$ are normalized to have unit Euclidean norm (this does not change convexity by Lemma~\ref{lem:rescale}). Since $-\MA$ has diagonal entries of zero (because of the normalization $\|u_i\|_2 = 1$) and off-diagonal entries that are strictly positive (by Cauchy-Schwarz), a celebrated result of Schoenberg ensures that $-\MA$ is conditionally NSD if and only if there is an embedding $z_1, \dots, z_n \in \ell_2$ such that the Euclidean distance $\|z_i - z_j\|_2 = \sqrt{- (\MA)_{ij}}$ for each $i,j \in \{1,\dots, n\}$~\citep[Theorem 1]{Schoenberg35}. In particular, this implies that if $\MA$ were conditionally PSD, then $\{\sqrt{- (\MA)_{ij}} \}_{i,j=1}^n$ satisfy the triangle inequality. However, 
\[
\sqrt{- (\MA)_{ij}} = \sqrt{-\log |u_i^Tu_j|}
\]
is precisely the Martin distance between the lines in directions $u_i$ and $u_j$, and the Martin distance is well-known to \emph{not} satisfy the triangle inequality (i.e. it is not a metric\footnote{E.g., the construction from Example~\ref{exa:not-psd}: $u_1 = e_1$, $u_2 = e_2$, and $u_3 = e_1 + e_2$. Then $\sqrt{-\log|u_1^Tu_3| } + \sqrt{-\log|u_2^Tu_3| } = 0 + 0 < \infty = \sqrt{-\log |u_1^Tu_2|}$, violating the triangle inequality.})~\citep[Section 12]{DezaDeza09}. Therefore $\MA$ cannot always be conditionally PSD.

\section{Optimizing the distribution}\label{sec:inapprox}
We now investigate the design problem from an algorithmic viewpoint. First we show in Subsection~\ref{subsec:inapprox} that this problem is $\NPclass$-hard in general; then in Subsection~\ref{subsec:ptas}, we briefly note that in a certain ``well-conditioned'' case, there is a polynomial-time approximation algorithm.

\subsection{Inapproximability}\label{subsec:inapprox}
The main result here (Theorem~\ref{thm:dec}) is that even the stabilization decision problem for Lyapunov exponents is $\NPclass$-hard. That is, unless $\Pclass = \NPclass$, there is no polynomial-time algorithm for deciding whether there exists a distribution $p \in \Delta_n$ satisfying $\lyap < 0$. Moreover, these hardness results hold even in the ``simple'' case where $\cA$ contains only rank-one symmetric matrices.

\begin{theorem}\label{thm:dec}
Given a collection $\cA$ of $n$ symmetric rank-one matrices, it is $\NPclass$-hard to decide whether
\begin{align}
\min_{p \in \Delta_n} \lyap
\label{eq:opt}
\end{align}
is negative.
\end{theorem}

A few remarks. First, note that Theorem~\ref{thm:dec} can be equivalently restated in terms of the Lyapunov spectral radius $R(\cA,p) = e^{\lyap}$ (see~\eqref{eq:def:explyap}). That is, it is $\NPclass$-hard to decide whether
\begin{align}
\min_{p \in \Delta_n} \explyap
\label{eq:opt-explyap}
\end{align}
is less than $1$ for a collection $\cA$ of $n$ symmetric rank-one matrices. Recall from the introduction that $\explyap < 1$ is equivalent to the (a.s.) asymptotic convergence of the stochastic linear dynamical system~\eqref{eq:lds}.
\par Next, we emphasize that this hardness result is somewhat surprising (at least to the authors) in the sense that the analogous optimization problem for the Joint Spectral Radius -- which is the deterministic kindred of the Lyapunov exponent (see the introduction for details) -- can be computed exactly in polynomial time for rank-one matrices~\citep[Remark 2.1]{AhmPar12}.
\par We now turn to proving Theorem~\ref{thm:dec}. Note that by Theorem~\ref{thm:formula}, it suffices to show the $\NPclass$-hardness of deciding the sign of the following quadratic optimization problem over the simplex:
\begin{align}
\min_{p \in \Delta_n} p^T \MA p,
\label{eq:opt-quad}
\end{align}
where $(\MA)_{ij} = \log \abs{u_i^Tu_j}$. We observed in the previous Section~\ref{sec:convexity} that, in general, this optimization problem~\eqref{eq:opt-quad} is not convex since $\MA$ is not conditionally PSD. Although this lack of convexity is a first step towards suggesting the hardness of the minimization problem~\eqref{eq:opt-quad}, it of course does not constitute a hardness proof in itself.
\par We will prove Theorem~\ref{thm:dec} via a reduction from approximating the Independent Set problem. It is helpful to first recall two classical results, namely the Motzkin-Straus formulation for the Independent Set problem as a quadratic optimization over the simplex, and the hardness of approximation of the Independent Set problem. Below, $\alpha(G)$ denotes the \emph{independence number} of a graph $G$, i.e. the size of the largest subset of nodes that has no edges between them. $A_G$ denotes the adjacency matrix of $G$, i.e. the entry $(A_G)_{ij}$ is $1$ if $(i,j)$ is an edge of $G$ and is $0$ otherwise.

\begin{theorem}\citep[Corollary 1]{MotStra65}
\label{thm:motzkin-strauss}
For any undirected graph $G$, $
\frac{1}{\alpha(G)}
= \min_{p \in \Delta_n} p^T(I + A_G)p$.
\end{theorem}

\begin{theorem}\citep[Theorem 1.1]{Zuckerman06}\footnote{The complexity result stated in~\citep{Zuckerman06} -- as well in much of the related literature -- is for the problem of computing the size of the largest clique in a graph. However, this problem is well-known to be equivalent to the independent set problem, since a set of nodes is a clique in a graph if and only if it is an independent set in the complement graph.}\label{thm:indep-inapprox}
For all $\eps > 0$, it is $\NPclass$-hard to approximate the independence number of a graph on $n$ nodes to within a multiplicative factor of $n^{1-\eps}$.
\end{theorem}

The reduction for proving Theorem~\ref{thm:dec} is as follows. Given an undirected graph $G$ on $n$ nodes and with edge set $E$, let $B := 3nI + \exp[I + A_G]$, where the exponential is applied entrywise. Note that $B$ is symmetric and diagonally dominant since for each $i \in [n]$, $B_{ii} = 3n + e \geq n + (e-1) \cdot \text{degree}(i) - 1= \sum_{j \neq i} B_{ij}$. Thus $B$ is PSD 
and in particular admits a factorization $B = U^TU$ for some $U \in \R^{d \times n}$, $d \leq n$. Denote the columns of $U$ by $u_1, \dots, u_n \in \R^d$, and let $M \in \Rnn$ be the matrix with entries $M_{ij} := \log |u_i^Tu_j| = \log |(U^TU)_{ij}| = \log |B_{ij}| = \log B_{ij}$. Now the key observation is that $\min_{p \in \Delta_n} p^TMp$ is an $O(\log n)$ multiplicative approximation to $\tfrac{1}{\alpha(G)}$:

\begin{lemma}\label{lem:ineq} $\frac{1}{\alpha(G)} \leq \min_{p \in \Delta_n} p^TMp \leq O(\log n) \frac{1}{\alpha(G)}$.
\end{lemma}
\begin{proof}
Since $M = (I + A_G) + (\log(3n + e) - 1) I$, we have
\begin{align}
I + A_G
\leq 
M
\leq
O(\log n) \left( I + A_G \right),
\label{eq:reduction-entrywise}
\end{align}
where the matrix inequalities are in the entrywise sense. Thus
\begin{align}
\min_{p \in \Delta_n} p^T(I + A_G)p
\leq 
\min_{p \in \Delta_n} p^TMp
\leq
O(\log n) \min_{p \in \Delta_n} p^T( I + A_G)p.
\label{eq:reduction-opt}
\end{align}
The proof is then complete by applying Theorem~\ref{thm:motzkin-strauss}.
\end{proof}

The last ingredient we will need for the proof of Theorem~\ref{thm:dec} is the simple but helpful observation that $\lyap$ additively shifts if all matrices in $\cA$ are scaled by the same constant. This follows immediately from the definition of the Lyapunov exponent -- or equivalently, by homogeneity of the Lyapunov spectral radius. 

\begin{obs}\label{obs:scaling}
Let $\cA$ be some collection of square matrices, and let $c \cA := \{c A : A \in \cA  \}$ for some positive scalar $c$. Then $\lyapc = \lyap + \log c$ for all $p \in \Delta_n$.
\end{obs}

\begin{proof}[Proof of Theorem~\ref{thm:dec}]
Assume there exists a polynomial-time algorithm for this decision problem. Now consider an undirected graph $G$ on $n$ nodes, and note that the aforementioned matrix $M$ and vectors $\{u_i\}_{i=1}^n$ are computable in polynomial time in $n$. Letting $\cA := \{ u_iu_i^T \}_{i=1}^n$, we have by Theorem~\ref{thm:formula} that $\lyap = p^TMp$. Thus by Lemma~\ref{lem:ineq} and the trivial bound $\alpha(G) \in [1,n]$, it holds that
\begin{align}
\min_{p \in \Delta_n} \lyap \in \left[n^{-1}, O(\log n)\right].
\label{eq:pf-inapprox:search}
\end{align}
By Observation~\ref{obs:scaling}, we may -- via binary search -- refine the confidence interval in~\eqref{eq:pf-inapprox:search} to size $n^{-1}$ using $O(\log (\tfrac{O(\log n) - n^{-1}  }{n^{-1}}))
= O(\log n)$ queries to the assumed decision algorithm on $c \cA = \{c A : A \in \cA \}$, for appropriate scalings $c$. This yields a polynomial-time\footnote{To be precise, this requires checking that the bit-complexity of the inputs to the decision algorithm is also polynomial in $n$; however, this is not an issue. Specifically, in light of Observation~\ref{obs:scaling} and the original confidence interval~\eqref{eq:pf-inapprox:search}, each scaling $c$ used in the binary search satisfies $\log c \in [n^{-1}, O(\log n)]$. Moreover, since we refine the search interval only to size $O(n^{-1})$, it suffices to truncate each $\log c$ to precision $O(n^{-1})$. Together, this implies 
that the bit complexity of the scaled matrices $c \cA$ is amplified from that of $\cA$ by at most a polynomial factor in $n$.}  approximation to $\min_{p \in \Delta_n} \lyap$ to additive error $n^{-1}$, and therefore by Lemma~\ref{lem:ineq} yields a polynomial-time approximation $\alpha(G)$ within a multiplicative factor of $O(\log n)$. But this is $\NPclass$-hard by Theorem~\ref{thm:indep-inapprox}.
\end{proof}

We remark that although the above reduction was for $\NPclass$-hardness of \emph{deciding the sign} of $\min_{p \in \Delta_n} \lyap$, it also demonstrates an $\NPclass$-hardness barrier for \emph{approximating the value} of $\min_{p \in \Delta_n} \lyap$. This is recorded briefly below. Note that we show hardness for approximating $\lyapmin$ to \emph{additive} error (rather than multiplicative) since this implies hardness for multiplicatively approximating the optimal Lyapunov spectral radius $\explyap = e^{\lyap}$, which is the relevant rate of the dynamical system~\eqref{eq:lds}.

\begin{theorem}\label{thm:inapprox}
For all $\eps > 0$, it is $\NPclass$-hard to approximate $\min_{p \in \Delta_n} \lyap$ to an additive error of $n^{-\eps}$ for a
collection $\cA$ of $n$ symmetric rank-one matrices.
\end{theorem}
\begin{proof}
We begin in the same way as in the proof of Theorem~\ref{thm:dec}: given an undirected graph $G$ on $n$ nodes, construct the matrix $M$ and vectors $\{u_i\}_{i=1}^n$, let $\cA := \{ u_iu_i^T \}_{i=1}^n$, and note by Theorem~\ref{thm:formula} that $\lyap = p^TMp$. Now if $\lyapmin$ can be approximated to an additive error of $n^{-\eps}$ in polynomial time, then by Lemma~\ref{lem:ineq} this yields a polynomial-time algorithm that approximates $\alpha(G)$ to a multiplicative factor of $n^{1-\eps} O(\log n)$. But this is $\NPclass$-hard by Theorem~\ref{thm:indep-inapprox}.
\end{proof}

\subsection{PTAS for ``well-conditioned'' sets of rank-one matrices}
\label{subsec:ptas}

Here, we note that there is a polynomial-time approximation scheme (PTAS) for approximating the design problem $\min_{p \in \Delta_n} \lyap$ in the case where the set of rank-one matrices $\cA$ is ``well-conditioned.'' (Note this does not conflict with the hardness results shown above, as explained below.)
\par The main idea is that since $\lyap$ is a quadratic form in $p$ in this rank-one case by Theorem~\ref{thm:formula}, we may use the well-known PTAS for minimizing polynomials over the simplex~\citep[Theorem 1.7]{KleLauPar06}. Applying their PTAS to this design problem yields, for any \emph{fixed} constant $\delta > 0$, a polynomial-time algorithm that, given any collection $\cA$ of rank-one matrices, outputs a distribution $\hat{p} \in \Delta_n$ satisfying
\begin{align}
\lyaphat - \lyapmin \leq \delta \left( \lyapmax - \lyapmin \right).
\label{eq:ptas-approx}
\end{align}
Note that here ``polynomial time'' means time that is polynomial in the input size (namely $n$ and the bit complexity of $\cA$), but \emph{not} in $\delta$ (see the discussion in~\citep[Section 1.4]{KleLauPar06}).
\par Intuitively, this approximation guarantee~\eqref{eq:ptas-approx} is good whenever the quantity $\lyapmax - \lyapmin$ governing the error bound is not very large. One concrete setup where this provably occurs is when the set $\cA = \{u_iu_i^T \}_{i=1}^n$ is ``well-conditioned'' in the sense that all pairwise angles $\measuredangle(u_i,u_j)$ are uniformly bounded away from $\tfrac{\pi}{2}$ -- or in words, all pairs of $u_i$ and $u_j$ are ``uniformly away'' from being orthogonal. This is made precise in the following result. For simplicity of exposition, we assume that the matrices in $\cA$ are symmetric (i.e. of the form $u_iu_i^T$) and normalized ($\|u_i\|_2 = 1$), but these assumptions can be easily removed.

\begin{theorem}\label{thm:ptas}
Fix any $\delta,\gamma > 0$.
There is an algorithm that, given any $\cA = \{u_iu_i^T \}_{i=1}^n$ such that each $\|u_i\|_2 = 1$ and $|\cos \measuredangle(u_i,u_j)| \geq \gamma$, outputs $\hat{p} \in \Delta_n$ satisfying
\begin{align}
\lyaphat \leq \min_{p \in \Delta} \lyap + \delta \log \frac{1}{\gamma}
\end{align}
in time that is polynomial in $n$ and the input size of $\cA$ (but not necessarily in $\delta$ or $\gamma$).
\end{theorem}

We make use of the following simple bounds.

\begin{lemma}\label{lem:min-max}
Let $\cA = \{u_iu_i^T \}_{i=1}^n$. Then:
\begin{itemize}
\item $\max_{p \in \Delta_n} \lyap = 2 \max_{i} \log \|u_i\|_2$.
\item $\min_{p \in \Delta_n} \lyap \geq \min_{i,j} \log |u_i^Tu_j|$.
\end{itemize}
\end{lemma}
\begin{proof}
By Theorem~\ref{thm:formula}, $\lyap = p^T\MA p$ where $(\MA)_{ij} = \log|u_i^Tu_j|$. For the maximization problem, note that $2 \max_{i} \log \|u_i\|$ is achieved by $p = e_i$. Moreover, this value is optimal since for any $p \in \Delta_n$,
\[
\lyap
= p^T\MA p
= \sum_{ij} p_ip_j (\MA)_{ij}
\leq \Big(\sum_{ij} p_ip_j\Big) \max_{ij} (\MA)_{ij}
= \max_{ij} (\MA)_{ij}
\leq \max_i (\MA)_{ii},
\]
where the final inequality is by the definition of $\MA$ and Cauchy-Schwarz. The bound on the minimization problem follows similarly: for any $p \in \Delta_n$,
\[
\lyap
= p^T\MA p
= \sum_{ij} p_ip_j (\MA)_{ij}
\geq \Big(\sum_{ij} p_ip_j\Big) \min_{ij} (\MA)_{ij}
= \min_{ij} (\MA)_{ij}.
\]
\end{proof}

\begin{proof}[Proof of Theorem~\ref{thm:ptas}]
By Lemma~\ref{lem:min-max},
$\lyapmax - \lyapmin
\leq
\max_{i} \log \left(\|u_i\|_2^2\right) - \min_{ij} \log \left( |u_i^Tu_j| \right)
\leq 
\log \frac{1}{\gamma}$. The proof is complete by~\eqref{eq:ptas-approx}.
\end{proof}

Note that the PTAS in Theorem~\ref{thm:ptas} is not at odds with the $\NPclass$-hardness results proved in Theorems~\ref{thm:dec} and~\ref{thm:inapprox}. This is because the runtime is \emph{not} polynomial in the accuracy parameter $\delta$, and to obtain the approximation bounds discussed in the hardness results, $\delta$ must be taken to be scaling with $n$ -- resulting in a runtime that is not polynomial in $n$.


\section{Extension to ergodic stationary processes}\label{sec:extensions}
So far in this paper, we have considered the analysis and design problems for the Lyapunov exponent in the standard setup where the random matrices are i.i.d. from a discrete distribution. In this section, we discuss extensions to the more general setup of ergodic stationary processes.
As corollaries, we discuss setups including i.i.d. processes from continuous distributions, exchangeable processes, and Markov processes.
This allows for the study of the stability of more exotic stochastic linear dynamical systems~\eqref{eq:lds}. The corresponding analysis and design problems are discussed in Subsections~\ref{subsec:extensions:stoc} and~\ref{subsec:extensions:des}, respectively. 

\subsection{Ergodic formula for rank-one matrices}\label{subsec:extensions:stoc}
Here we show that, with only a minor modification, the rank-one Lyapunov exponent formula in Theorem~\ref{thm:formula} for i.i.d. processes extends to ergodic stationary processes. The proof is similar, the primary difference being the use of Birkhoff's Ergodic Theorem.
\par Below, let $\cA = \{uv^T : u,v \in \R^d \}$, and let it be equipped with the Borel $\sigma$-algebra (arising equivalently from any norm on $\R^{d \times d}$). We consider processes $A_{\sigma_1}, A_{\sigma_2}, \dots$ taking values in $\cA$. Recall that this process is said to be \emph{stationary} if $(A_{\sigma_1}, \dots, A_{\sigma_k})$ has the same distribution as $(A_{\sigma_{\tau+ 1}}, \dots, A_{\sigma_{\tau + k}})$ for every $k,\tau \in \N$. Recall also that this process is said to be \emph{ergodic} if every shift-invariant event has probability $0$ or $1$. We refer the reader to e.g.,~\citep[Chapter 7]{Durrett10} for further background on such processes.

\begin{theorem}\label{thm:stoc}
Let $A_{\sigma_1},A_{\sigma_2}, \dots$ be an ergodic stationary process taking values in $\cA$, and let $P$ denote the marginal distribution of $(A_{\sigma_1}, A_{\sigma_2})$. If the integrability conditions $\E |\log \|A_{\sigma_1}\|| < \infty$ and $\int \abs{\log\abs{u^Tv'}} dP(uv^T,u'v'^T)
 < \infty$ are satisfied,
then the (random) quantity
\begin{align*}
\lim_{k \to \infty} \frac{1}{k} \log \left\| A_{\sigma_k} \cdots A_{\sigma_1} \right\|
\end{align*}
exists a.s., and moreover converges to
\begin{align}
\frac{1}{k} \log \left\| A_{\sigma_k} \cdots A_{\sigma_1} \right\|
\to
\int \log\abs{u^Tv'} dP(uv^T,u'v'^T)
\label{eq:thm:formula-stoc}
\end{align}
a.s. and in $L_1$.
\end{theorem}

Note that the formula~\eqref{eq:thm:formula-stoc} is independent of the decompositions $u^Tv$ of the rank-one matrices. This follows by a simple calculation that is identical to the one in Remark~\ref{rem:skew-symm}.

\begin{proof}
%
Note that the theorem statement is independent of the norm $\|\cdot\|$ by equivalence of finite-dimensional norms; for convenience of computation, let us take it to be the spectral norm. Consider rank-one decompositions $A_{\sigma_k} = u_{\sigma_k}v_{\sigma_k}^T$ where w.l.o.g. we take $\|v_{\sigma_k}\|_2 = 1$. Just as in~\eqref{eq:pf-thm-formula:expand}, expand the (random) quantity
\begin{align}
\frac{1}{k} \log \norm{A_{\sigma_k} \cdots A_{\sigma_2} A_{\sigma_1} }_2
= \frac{1}{k} \sum_{t=1}^{k-1} \log \abs{u_{\sigma_t}^Tv_{\sigma_{t+1}}}
+ \frac{1}{k}\log \|u_{\sigma_k}\|_2.
\label{eq:pf-thm-stoc:expand}
\end{align}
To deal with the final term $\frac{1}{k}\log \|u_{\sigma_k}\|_2$, observe that by stationarity and the first integrability assumption, we have $\E |\log \norm{u_{\sigma_k}}_2| = \E |\log \norm{u_{\sigma_k}}_2 + \log \norm{v_{\sigma_k}}_2| = \E |\log \norm{A_{\sigma_k}}| = \E | \log \norm{A_{\sigma_1}}| < \infty$.
Thus the random variable $\log \|u_{\sigma_k}\|_2$ is bounded in $L_1$ (and thus also bounded a.s.), and so
\begin{align}
\frac{1}{k} \left(\log \normlr{u_{\sigma_k}}_2 + \log \normlr{v_{\sigma_1}}_2\right) \to 0,
\label{eq:pf-thm-stoc:term}
\end{align}
where the convergence is a.s. and in $L_1$. Now by ergodicity, stationarity, and the second integrability assumption, we may apply Birkhoff's Ergodic Theorem (see e.g.~\citep[Theorem 7.2.1]{Durrett10}) to obtain
\begin{align}
\frac{1}{k}  \sum_{t=1}^{k} \log \abs{u_{\sigma_t}^Tv_{\sigma_{t+1}}}
\to
\int \log\abs{u^Tv'} dP(uv^T,u'v'^T),
\label{eq:pf-thm-stoc:birk}
\end{align}
where the convergence is a.s. and in $L_1$. Combining~\eqref{eq:pf-thm-stoc:expand},~\eqref{eq:pf-thm-stoc:term}, and~\eqref{eq:pf-thm-stoc:birk} completes the proof.
\end{proof}

We remark that the ergodic formula~\eqref{eq:thm:formula-stoc} in Theorem~\ref{thm:stoc} depends on the process only through its \emph{second-order statistics}. This occurs since the matrices are all rank-one. We now discuss several examples of ergodic stationary processes for which Theorem~\ref{thm:stoc} applies.

\subsubsection{I.i.d. processes}\label{subsubsec:extensions:iid}
A first, simple example of ergodic stationary processes is an i.i.d. process. Specifically, consider the setup where $A_{\sigma_1}, A_{\sigma_2}, \dots$ are i.i.d. random variables with (the same) distribution $\mu$ over $\cA$. Then Theorem~\ref{thm:stoc} shows that, assuming some mild integrability conditions, we have the (a.s. and $L_1$) convergence:
\begin{align}
\frac{1}{k} \log \normlr{ A_{\sigma_k}\cdots A_{\sigma_1}}
\to
\int \int \log |u^Tv'| \, d\mu(uv^T) d\mu(u'v'^T)
\label{eq:extensions:iid}
\end{align}
Note that this holds for continuous distributions $\mu$, i.e. distributions with infinite support. In the special case that $\mu$ has finite support, we recover Theorem~\ref{thm:formula}.

\begin{example}\label{ex:cont}
Let $S^{d-1} = \{z \in \R^d : \|z\|_2 = 1 \}$ denote the $d-1$ sphere, for $d > 1$, and let $\UnifSd$ denote the uniform distribution over it (i.e. the $d-1$ dimensional spherical measure). In this example, we explicitly compute the Lyapunov exponent over ``normalized'' rank-one matrices $\cA := \{uv^T : u,v\in S^{d-1} \}$ for the following two probability measures:
\begin{itemize}
\item $\musym$ is the induced distribution of $uu^T \in \cA$, where $u \sim \UnifSd$.
\item $\muasym$ is the induced distribution of $uv^T \in \cA$, where $u,v \sim \UnifSd$ and are independent.
\end{itemize}
By Theorem~\ref{thm:stoc},
the Lyapunov exponent is the same for both $\musym$ and $\muasym$, and is equal to:
\begin{align}
\lyapmusym = \lyapmuasym = \E_{u,v \sim \UnifSd} \log \abs{u^Tv}.
\label{eq:ex-cont:expect}
\end{align}
To compute this, we first compute the distribution $\nu$ of $u^Tv$. By rotational invariance of the spherical measure $\UnifSd$, $\nu$ is equal to the distribution of $z_1$ where $z \sim \UnifSd$. This distribution is well-known to have density~\citep{Szego39}
\begin{align}
\frac{d\nu(t)}{dt}
=
\frac{\Gamma\left(\tfrac{d}{2}\right)}{\sqrt{\pi} \Gamma\left(\tfrac{d-1}{2}\right)}
\left( 1 - t^2 \right)^{(d-3)/2}
\mathds{1}_{t \in (-1,1)},
\label{eq:ex-cont:dist}
\end{align}
where $\Gamma(\cdot)$ denotes the Gamma function. Using this, the integral in the right hand side of~\eqref{eq:ex-cont:expect} -- and thus the Lyapunov exponents for both $\musym$ and $\muasym$ -- can be computed to be:
\begin{align}
\E_{u,v \sim \UnifSd} \log \abs{u^Tv}
&=
\frac{\Gamma\left(\tfrac{d}{2}\right)}{\sqrt{\pi} \Gamma\left(\tfrac{d-1}{2}\right)} \int_{-1}^{1} \left( 1 - t^2 \right)^{(d-3)/2} \log |t| dt
\nonumber \\ &=
- \frac{\Psi\left( \tfrac{d}{2} \right) + \gamma + \log 4 }{2},
\label{eq:ex-cont:final}
\end{align}
where $\Psi(\cdot)$ denotes the Digamma function and $\gamma \approx 0.5772$ denotes the Euler-Mascheroni constant.
\par The value of~\eqref{eq:ex-cont:final} is reported in Table~\ref{tab:ex-cont} for small values $d \in \{2, \dots, 8\}$. Note that the value for $d=3$ is equal to $-1$; this simple value arises because of the well-known fact that the distribution $\nu$ of the first coordinate of a random vector on the $2$-sphere, is the uniform distribution over $[-1,1]$ (this can be seen directly from~\eqref{eq:ex-cont:dist}). Indeed, using this one can easily derive the special case
$\E_{u,v \sim \UnifStwo}\log |u^Tv| = \half \int_{-1}^1 \log |z| dz = \int_0^1 \log z dz = (z\log z - z)_0^1 = -1$.
\par The asymptotic behavior of~\eqref{eq:ex-cont:final} as the dimension $d \to \infty$ is
\begin{align}
\lyapmusym = \lyapmuasym = -\half (\log d + \gamma + \log 2) + O(d^{-1}).
\label{eq:ex-cont:asymp}
\end{align} 
due to the well-known inequality $\log z - z^{-1} \leq \Psi(z) \leq \log z - (2z)^{-1}$ that holds for all $z > 0$~\citep{Alzer97}. Note that the leading term $\log(\tfrac{1}{\sqrt{d}})$ makes sense intuitively since if $u,v \sim \UnifSd$ are independent, then $|u^Tv|$ concentrates tightly around $\tfrac{1}{\sqrt{d}}$.\footnote{Here is a simple, standard way to derive this informally. The starting point is that if $\tilde{u} \sim \cN(\zero, I)$, the normal distribution centered at the origin and with identity covariance, then $\tilde{u}/\|\tilde{u}\|_2 \sim \UnifSd$ by rotational invariance. Thus $u^Tv$ is equal in distribution to $\tilde{u}^T\tilde{v}/(\|\tilde{u}\|_2\|\tilde{v}\|_2)$ where $u$ and $v$ are i.i.d. $\cN(\zero, I)$. Now the numerator $\tilde{u}^T\tilde{v}$ is equal in distribution to $\tilde{v}_1 \sim \cN(0,1)$ by rotational invariance, and the denominator $\|\tilde{u}\|_2\|\tilde{v}\|_2$ concentrates tightly around $\sqrt{d}$ since $\|\tilde{u}\|_2^2 = \sum_{i=1}^d \tilde{u}_i^2$ is $\Theta(\sqrt{d})$ by the Law of Large Numbers. Thus $|u^Tv|$ is approximately equal to $|\tilde{v}_1|/\sqrt{d}$, where $\tilde{v}_1 \sim \cN(0,1)$, and this concentrates tightly around $\Theta(1/\sqrt{d})$. The dropped terms in this discussion are given precisely in the formula~\eqref{eq:ex-cont:final} and asymptotically by~\eqref{eq:ex-cont:asymp}.}
\end{example}

\begin{table}[t]
\centering
\caption{The Lyapunov exponents $\lyapmusym = \lyapmuasym = - \half (\Psi( \tfrac{d}{2} ) + \gamma + \log 4)$, as derived in Example~\ref{ex:cont}, for dimensions $d \in \{2, \dots, 8\}$. Numerical values reported to six significant digits.}
\label{tab:ex-cont}
\begin{tabular}{|c|c|}
\hline
\textbf{d} & $- \half (\Psi( \tfrac{d}{2} ) + \gamma + \log 4)$ \\ \hline
2                                & $-0.693147$                              \\ \hline
3                                & $-1       $                              \\ \hline
4                                & $ -1.19315$                              \\ \hline
5                                & $ -1.33333$                              \\ \hline
6                                & $ -1.44315$                              \\ \hline
7                                & $ -1.53333$                              \\ \hline
8                                & $ -1.60981$                              \\ \hline
\end{tabular}
\end{table}
%

\subsubsection{Exchangeable processes}\label{subsubsec:extensions:exch}
Although exchangeable processes are \emph{not} ergodic stationary processes\footnote{A simple example of an exchangeable process that is not ergodic is the process $X_1, X_2, \dots$ where all random variables take the same value $X$, and $X$ is a Bernoulli random variable.}, we can understand them in terms of such processes using De Finetti's Theorem~\citep{DeFinetti38}. Specifically, consider the setup of an exchangeable process $A_{\sigma_1}, A_{\sigma_2}, \dots$ taking values in $\cA$. Recall that (infinite) exchangeability means that for each $k \in \N$ and each permutation $\pi$ of $\{1, \dots, k\}$, $(A_{\sigma_1}, \dots, A_{\sigma_k})$ has the same distribution as $(A_{\pi(\sigma_1)}, \dots, A_{\pi(\sigma_k)})$. Now De Finetti's Theorem -- or, more precisely, ~\citep{HewSav55}'s generalization thereof to complete, separable metric spaces -- states that the law of this process is equal to a mixture $\nu$ over i.i.d. product distributions over $\cA$. It therefore follows from Theorem~\ref{thm:stoc} that under the same mild integrability conditions, we have the (a.s. and $L_1$) convergence of
\begin{align}
\frac{1}{k} \log \normlr{ A_{\sigma_k}\cdots A_{\sigma_1}}
\to
\int \int \int \log |u^Tv'|\, d\mu(uv^T) d\mu(u'v'^T) d\nu(\mu).
\label{eq:extensions:exchangeable}
\end{align}

\subsubsection{Markov processes}\label{subsubsec:extensions:markov} Another example of an ergodic stationary process is a Markov process over a countable state space (for simplicity) with an irreducible transition kernel (for ergodicity) and with an initial distribution equal to the stationary distribution (for stationarity), see e.g.~\citep[Example 7.1.7]{Durrett10}. 
That is, consider the setup where the non-initial matrices $A_{\sigma_2}, A_{\sigma_3}, \dots$ are drawn from an irreducible Markov kernel $Q$ over a countable subset of the rank-one matrices $\cA$, and the initial matrix $A_{\sigma_1}$ is drawn from the stationary distribution $\pi_Q$ of the chain.
Then Theorem~\ref{thm:stoc} guarantees that, under mild integrability asumptions, we have the (a.s. and $L_1$) convergence
\begin{align}
\frac{1}{k} \log \normlr{ A_{\sigma_k}\cdots A_{\sigma_1}}
\to
\int \int \log |u^Tv'|\, \pi_Q(uv^T) Q(uv^T,u'v'^T).
\label{eq:extensions:markov}
\end{align}
Note that this recovers Theorem~\ref{thm:formula} since the original Lyapunov exponent setup is the special case where the columns of $Q$ are constant, i.e. $Q = \bone p^T$, and the stationary distribution is $\pi_Q = p$. We also remark that~\eqref{eq:extensions:markov} still holds when the assumption that the initial matrix $A_{\sigma_1}$ has stationary distribution $\pi_Q$ is relaxed.\footnote{This can be proved using a nearly identical argument as in the proof of Theorem~\ref{thm:formula}.}

\subsection{Optimizing the Lyapunov exponent}\label{subsec:extensions:des}
We now turn to the design problem of optimizing the convergence rate of the stochastic linear dynamical system~\eqref{eq:lds} over different types of stochastic processes $A_{\sigma_1},A_{\sigma_2}, \dots$. We already showed in Subsection~\ref{subsec:inapprox} that this design problem is $\NPclass$-hard for i.i.d. processes. Below, we show that the design problem is also $\NPclass$-hard for exchangeable processes (Subsection~\ref{subsubsec:des:exch}) yet admits a polynomial-time algorithm for Markov processes (Subsection~\ref{subsubsec:des:markov}).

\subsubsection{Exchangeable processes}\label{subsubsec:des:exch}
Here, we show that the design problem is $\NPclass$-hard over exchangeable processes:

\begin{theorem}\label{thm:des-exchangable}
Given a collection $\cA$ of $n$ symmetric rank-one matrices, it is $\NPclass$-hard to decide whether there exists an exchangeable process $A_{\sigma_1}, A_{\sigma_2}, \dots$ taking values in $\cA$, for which the stochastic linear dynamical system~\eqref{eq:lds} converges (a.s.) for every initial state.
\end{theorem}

To give intuition for this result, recall the following three facts: (i) by De Finetti's Theorem, exchangeable processes are mixtures of i.i.d. distributions; (ii) the Lyapunov exponent is a linear functional in this mixing measure (see~\eqref{eq:extensions:exchangeable}); and (iii) every linear program with a compact feasible set admits an optimal extreme point. Combining these three facts implies that, among exchangeable processes, there is an optimal i.i.d. process. This explains the hardness in Theorem~\ref{thm:des-exchangable}, since the design problem over i.i.d. processes is $\NPclass$-hard by Theorem~\ref{thm:dec}. We now formalize this intuition:


\begin{proof}[Proof of Theorem~\ref{thm:des-exchangable}]
By~\eqref{eq:extensions:exchangeable},
there exists such a stabilizing exchangeable process if and only if there exists such a stabilizing i.i.d. process. But deciding whether there exists such a stabilizing i.i.d. process is $\NPclass$-hard by Theorem~\ref{thm:dec}.
\end{proof}

\subsubsection{Markov processes}\label{subsubsec:des:markov}
Here we show that the design problem is significantly easier when optimizing over irreducible Markov processes: there is a polynomial-time algorithm that not only can detect whether the associated dynamical system~\eqref{eq:lds} is stabilizable, but also can compute an irreducible Markov process for which the convergence is fastest. This is in stark contrast to the i.i.d. and exchangeable setups, for which even the stabilization decision problem is $\NPclass$-hard (see Theorems~\ref{thm:dec} and~\ref{thm:des-exchangable}).
\par It will be helpful to introduce some notation. We say a square matrix is \emph{irreducible} if for every simultaneous permutation of its rows and columns, the resulting matrix is a block-diagonal matrix with at most one diagonal block, not counting zero diagonal blocks\footnote{We note that the standard definition of an irreducible matrix does count zero diagonal blocks; this slight tweak allows us to more easily consider optimizing over subsets of $\cA$ without having to introduce additional notation. A similar remark goes for our definition of irreducible Markov transition matrices.}. We say $Q \in \Rpnn$ is an \emph{irreducible Markov transition matrix} on $n$ states if $Q$ is irreducible and $Q\bone = \bone$; let $\cQ_n$ denote the set of all such matrices. Each $Q \in \cQ_n$ has a unique stationary distribution which we will denote by $\pi_Q$. 
\par For simplicity, we assume below an exact arithmetic model. Inexact arithmetic is easily handled by computing logarithms to arbitrary precision.

\begin{theorem}\label{thm:des-markov}
There is a polynomial-time algorithm that, given any collection of $n$ rank-one matrices $\cA = \{u_iv_i^T\}_{i=1}^n$, outputs a minimizer of
\begin{align}
\min_{Q \in \cQ_n} \sum_{i,j=1}^n (\pi_Q)_i Q_{ij} \log |u_i^Tv_j|.
\label{eq:des-markov}
\end{align}
\end{theorem}
Note that the objective function in~\eqref{eq:des-markov} is precisely the convergence rate of the corresponding dynamical system~\eqref{eq:lds} by the discussion in Subsection~\ref{subsubsec:extensions:markov}. It therefore follows immediately that stability of the dynamical system can be detected in polynomial time: simply run the algorithm in Theorem~\ref{thm:des-markov} and check whether the optimization problem~\eqref{eq:des-markov} has strictly negative value.

\begin{cor}\label{cor:des-exchangeable}
There is a polynomial-time algorithm that, given a collection $\cA$ of $n$ rank-one matrices, decides whether there exists an irreducible Markov process $A_{\sigma_1}, A_{\sigma_2}, \dots$ taking values in $\cA$ such that the stochastic linear dynamical system~\eqref{eq:lds} converges (a.s.) for every initial state.
\end{cor}

We now turn to the proof of Theorem~\ref{thm:des-markov}. We proceed in two steps, informally sketched as follows. First, we observe that the optimization problem~\eqref{eq:des-markov} can be equivalently recast as a linear program over circulations (Lemma~\ref{lem:markov-circ}). Then, by a basic result about decomposing circulations into cycles, we observe that this linear program is equivalent to the Minimum-Cycle-Mean problem in a certain associated graph (Lemma~\ref{lem:markov-ext}). The proof is then completed by recalling a well-known polynomial-time algorithm for the Minimum-Cycle-Mean problem.
\par We now formalize step $1$. Below, we say $F \in \Rpnn$ is a (normalized) \emph{circulation} of size $n$ if $F \bone = F^T\bone$ and $\sum_{ij}F_{ij} = 1$; let $\cF_n$ denote the set of all such matrices.

\begin{lemma}\label{lem:markov-circ}
$
\min_{Q \in \cQ_n} 
\sum_{ij} (\pi_Q)_i Q_{ij} \log |u_i^Tv_j|
=
\min_{F \in \cF_n} 
\sum_{ij} F_{ij} \log |u_i^Tv_j|
$.
\end{lemma}
\begin{proof}
For any $Q \in \cQ_n$, the matrix $F$ with entries $F_{ij} := (\pi_Q)_i Q_{ij}$ is easily checked to be in $\cF_n$ since $F\bone = F^T\bone = \pi_Q$ by definition of $\pi_Q$ being the stationary distribution for $Q$. This proves the inequality ``$\geq$''. For the other direction ``$\leq$'', consider an optimal $F \in \cF_n$ for the optimization problem on the right hand side. We may assume w.l.o.g. that $F$ is irreducible since a circulation can always be decomposed as a convex combination of irreducible circulations (this follows e.g. from~\citep[Exercise 7.14]{BerTsi97}), and each of these irreducible circulations must also be optimal by linearity of the cost function $F \mapsto \sum_{ij} F_{ij} \log |u_i^Tv_j|$.
Let $\pi := F\bone = F^T\bone$, let $j^*$ be any fixed index in $\{1, \dots, n\}$ such that $\pi_{j^*} > 0$ (such an index exists since $F \in \cF_n$), and let $Q$ be the matrix with entries
\begin{align}
Q_{ij} := \begin{cases}
F_{ij} / \pi_i & \pi_i > 0 \\
1 & \pi_i = 0, j = j^* \\
0 & \pi_i = 0, j \neq j^*
\end{cases}.
\label{eq:lem:markov-circ}
\end{align}
It is straightforward to check that $Q \in \cQ_n$, that it has stationary distribution $\pi_Q = \pi$, and that $(\pi_Q)_i Q_{ij} = F_{ij}$ for all $i,j$. This completes the proof.
\end{proof}

We now formalize step $2$. Below, $G_{\cA}$ denotes the complete directed graph on $n$ nodes, one for each matrix $A_i$, with weight $\log |u_i^Tv_j|$ on the edge from node $i$ to node $j$.  The \emph{mean weight} of a cycle in a $G_{\cA}$ is the sum of the weights of the edges in the cycle, divided by the total number of edges.
The \emph{minimum-cycle-mean} of $G_{\cA}$ is the minimum mean weight of a cycle, where the minimum is over all cycles in $G_{\cA}$.

\begin{lemma}\label{lem:markov-ext}
$\min_{F \in \cF_n} 
\sum_{ij} F_{ij} \log |u_i^Tv_j|$ is equal to the minimum-cycle-mean of $G_{\cA}$.
%
\end{lemma}
\begin{proof}
Recall that a circulation (flow) can be decomposed as the convex combination of cycle (flows) see e.g.~\citep[Exercise 7.14]{BerTsi97}. Thus by linearity of the objective function $F \mapsto \sum_{ij} F_{ij} \log |u_i^Tv_j|$, there exists an optimal circulation $F \in \cF_n$ and a cycle $C$ such that $F_{ij}$ is equal to $1/|C|$ if $(i,j)$ is an edge in $C$, and $0$ otherwise. The direction ``$\geq$'' then follows since $\sum_{ij}F_{ij}\log |u_i^Tv_j|$ is equal to the mean weight of $C$. The other direction ``$\leq$'' follows by constructing $F \in \cF_n$ from an optimal cycle $C$ in the same way.
\end{proof}

\par We are now ready to conclude the proof of Theorem~\ref{thm:des-markov}.

\begin{proof}[Proof of Theorem~\ref{thm:des-markov}]
By Lemmas~\ref{lem:markov-circ} and~\ref{lem:markov-ext}, $\min_{Q \in \cQ_n} \sum_{ij} (\pi_Q)_i Q_{ij} \log |u_i^Tv_j|$ is equal to the minimum-cycle-mean of $G_{\cA}$. An optimal such cycle (i.e. one with minimum mean weight) can be computed in $O(n^3)$ time~\citep{Karp78}. Let $F$ denote the $n \times n$ matrix with entries $F_{ij}$ equal to $1/|C|$ if the cycle contains the edge $(i,j)$ and equal to $0$ otherwise. Clearly $F \in \cF_n$ and $\sum_{ij} F_{ij} \log |u_i^Tv_j|$ is equal to the mean weight of $C$. Now construct $Q \in \cQ_n$ from $F$ by~\eqref{eq:lem:markov-circ} so that $\sum_{ij} F_{ij} \log |u_i^Tv_j| = \sum_{ij} (\pi_Q)_i Q_{ij} \log |u_i^Tv_j|$. Thus $Q$ is an optimal solution for~\eqref{eq:des-markov}. Since $F$ and $Q$ are clearly computable from $C$ in polynomial time, this completes the proof.
\end{proof}


\begin{remark}
It is known that the Joint Spectral Radius of a collection of rank-one matrices can be computed from the minimum-cycle-mean of $G_{\cA}$~\citep[Theorem 2.2]{AhmPar12}. We have shown here that the analogous optimization problem over irreducible Markov processes is also equivalent to the minimum-cycle-mean of $G_{\cA}$.
\end{remark}

\bibliographystyle{abbrv}
\bibliography{lyap_rank1_bib}{}

\end{document}

%% file: lyap_rank1_preamble.tex
\usepackage{amsmath,amssymb,amsthm}
\usepackage{graphicx,color}
\usepackage[hyphens]{url}
\usepackage{dsfont}
\usepackage{booktabs}
\usepackage[square,sort,numbers]{natbib}
\usepackage[sort,nocompress,noadjust]{cite}
\usepackage[normalem]{ulem}
\usepackage[boxed]{algorithm}
\usepackage{mathtools}
\usepackage[noend]{algpseudocode}

\usepackage{subcaption}
\usepackage{tikz}

\numberwithin{equation}{section}

\newtheorem{theorem}{Theorem}[section]
\newtheorem{cor}{Corollary}[section]
\newtheorem{lemma}{Lemma}[section]
\newtheorem{remark}{Remark}[section]

\newtheorem{obs}{Observation}[section]

\theoremstyle{definition}

\newtheorem{example}{Example}

\newcommand{\cA}{\mathcal{A}}

\newcommand{\cF}{\mathcal{F}}

\newcommand{\cN}{\mathcal{N}}

\newcommand{\cQ}{\mathcal{Q}}

\newcommand{\R}{\mathbb{R}}

\newcommand{\N}{\mathbb{N}}

\newcommand{\Rnn}{\R^{n \times n}}

\newcommand{\Rpn}{\R_{\geq 0}^n}
\newcommand{\Rpnn}{\R_{\geq 0}^{n \times n}}

\newcommand*{\E}{\mathbb{E}}
\newcommand{\supp}{\operatorname{supp}}

\newcommand{\aseq}{\overset{\text{a.s.}}{=}}

\newcommand{\bone}{\mathbf{1}}
\newcommand{\zero}{\mathbf{0}}

\newcommand*{\eps}{\varepsilon}

\DeclareMathOperator*{\argmin}{argmin}

\renewcommand{\leq}{\leqslant}
\renewcommand{\geq}{\geqslant}
\DeclareMathOperator{\half}{\frac{1}{2}}

\providecommand{\abs}[1]{\left\lvert#1\right\rvert}

\providecommand{\norm}[1]{\lVert{#1}\rVert}
\providecommand{\normlr}[1]{\left\|{#1}\right\|}

\newcommand{\explyap}{R(\cA, p)}
\newcommand{\lyap}{\lambda(\cA, p)}

\newcommand{\jsr}{\rho(\cA)}
\newcommand{\MA}{M_{\cA}}

\newcommand{\lyapc}{\lambda(c \cA, p)}

\newcommand{\lyaphat}{\lambda(\cA, \hat{p})}

\newcommand{\lyapmin}{\min_{p \in \Delta_n} \lambda(\cA, p)}
\newcommand{\lyapmax}{\max_{p \in \Delta_n} \lambda(\cA, p)}

\newcommand{\UnifSd}{\mathsf{Unif}(S^{d-1})}
\newcommand{\UnifStwo}{\mathsf{Unif}(S^{2})}

\newcommand{\musym}{\mu_{\text{sym}}}
\newcommand{\muasym}{\mu_{\text{asym}}}
\newcommand{\lyapmusym}{\lambda(\cA, \musym)}
\newcommand{\lyapmuasym}{\lambda(\cA, \muasym)}

\newcommand{\Pclass}{\textsc{P}}
\newcommand{\NPclass}{\textsc{NP}}